\documentclass[11pt]{article}

\usepackage{latexsym,amsmath,amscd,amssymb,amsthm,graphics}
\usepackage{enumerate,epsfig}
\usepackage[margin=2cm]{geometry}
\usepackage{float}
\usepackage{graphicx}
\usepackage{framed,url,multicol,color}
\usepackage{shadow,enumerate}
\usepackage{makeidx}
\usepackage[colorlinks]{hyperref}
\usepackage[all]{xy}

\makeatletter

\usepackage[english]{babel}
\usepackage[T1]{fontenc}
\usepackage[utf8]{inputenc}
\usepackage{microtype}
\usepackage{tocbibind}
\usepackage{bbm}
\usepackage{color}
\usepackage[usenames,dvipsnames,svgnames,table]{xcolor}
\usepackage{graphicx}
\usepackage{array}
\usepackage{float}
\usepackage{pdfpages}
\usepackage[margin=2cm]{geometry}
\usepackage{stmaryrd} 

\begin{document}

\newtheorem{theorem}{Theorem}[section]
\newtheorem{definition}[theorem]{Definition}
\newtheorem{lemma}[theorem]{Lemma}
\newtheorem{remark}[theorem]{Remark}
\newtheorem{proposition}[theorem]{Proposition}
\newtheorem{corollary}[theorem]{Corollary}
\newtheorem{example}[theorem]{Example}

\setlength\parindent{0pt}

%
%
%
%
%
%

\title{Homogenization of a 1D pursuit law with delay and a counter-example}

\author{J\'er\'emy FIROZALY, Université Paris-Est, Cermics (ENPC), F-77455 Marne-la-Vallée }

\maketitle

\begin{abstract}
In this paper, we consider a one dimensional pursuit law with delay which is derived from traffic flow modelling. It takes the form of an infinite system of first order coupled delayed equations. Each equation describes the motion of a driver who interacts with the preceding one, taking into account his reaction time. We derive a macroscopic model, namely a Hamilton-Jacobi equation, by a homogenization process for reaction times that are below an explicit threshold. The key idea is to show, that below this threshold, a strict comparison principle holds for the infinite system. In a second time, for well-chosen dynamics and higher reaction times, we show that there exist some microscopic pursuit laws that do not lead to the previous macroscopic model. 
\end{abstract}

\section{Introduction}

In the present paper we consider an infinite system of delay differential equations (DDEs). This form can especially be derived from a one dimensional pursuit law on a straight road with a very simple model. Such models, where we follow the time position of every driver are called microscopic models. The effect of drivers' reaction times on the stability of the corresponding systems of DDEs is studied in \cite{atay2010complex} and \cite{nagatani2002physics}. Presentations of some classical microscopic traffic flow models of first order or second order are made in \cite{chand}, \cite{gipps} or \cite{costeseque}. Conversely, macroscopic models describe the time evolution of densities in global traffic flow. We want to study if we can recover a macroscopic model from our microscopic one. Formally, it corresponds to making the interdistance between drivers go to zero or to observe the pursuit law from a height and time that go to infinity. When possible, the passage can be made rigorous by an homogenization process that will be done here in the framework of viscosity solutions of Hamilton-Jacobi equations. The interested reader is referred to the pionnering works \cite{LPV} or \cite{Evans} about homogenization of Hamilton-Jacobi equations. He can see the works of \cite{light} or \cite{richards} for the historical approach used to define and analyse macroscopic traffic flow models and links with fluid mechanics. Several examples (and counter-examples) will be exhibited throughout this paper to show that both the initial dynamics and reaction time values will have a huge influence on the homogenization process. 

This paper derives from R\'egis Monneau's work \cite{monne} (see also \cite{costeseque}).

\subsection{Description of the model and main results}
The common velocity of each driver is supposed to be a Lipschitz continuous, bounded, nondecreasing function, $F$ say, of the distance that separates each driver from the preceding one (with $F(0)=0$). Each driver has its own reaction time and we assume that the reaction times are uniformly bounded from above and from below; we denote $\tau$ the upper bound and $\xi>0$ the lower bound. 
Let us study the evolution of the system during a time $T\in(0,+\infty]$ with $T>\tau$.

Given a sequence $(X_i)_{i\in \mathbb{Z}}$ of drivers’ positions on the road, the microscopic model then takes the form of an infinite system of DDEs:
\begin{align} \label{microdelay}
\frac{dX_i}{dt}(t)=F(X_{i+1}(t-\tau_i)-X_{i}(t-\tau_i))  \hspace{3 cm} (i,t)\in \mathbb{Z}\times (0,T),
\end{align}
where $\tau_i$ denotes the individual reaction time.

If we suppose that we know the dynamics of the cars during the initial time interval, we can define time functions $(x_i^0)_{i\in\mathbb{Z}}$ such that: 
\begin{align}\label{microinitial}
X_i(t)=x_i^0(t)   \qquad \forall t\in [-2\tau,0].
\end{align}

The first step is to embed the microscopic system dynamics into a single PDE. In order to go from the microscopic scale to the macroscopic one, we have to introduce a small parameter, $\varepsilon>0$ say, so as to rescale properly the function and to take into account the microscopic space and time oscillations with high frequency.

Hence, by hyperbolic change of variables, we define the function $u^\varepsilon$ as follows (the space variable will be denoted as $x$):
\begin{align}
u^\varepsilon(x,t)=\varepsilon X_{\lfloor \frac{x}{\varepsilon} \rfloor}(\frac{t}{\varepsilon}).
\end{align}
We suppose that there exists a Lipschitz continuous function $u_0: \mathbb{R} \times [-2\tau,0] \rightarrow \mathbb{R}$ such that: $$ x_i^0(t)=\frac{u_0(i\varepsilon, t \varepsilon)}{\varepsilon} \hspace{3cm} (i,t)\in \mathbb{Z}\times [-2\tau,0],$$
and a function $\tau_0: \mathbb{R} \rightarrow [\xi,\tau]$ such that: $$\tau_i=\tau_0(i\varepsilon) \qquad \forall i\in\mathbb{Z}. $$
The rescaled model can thus be embedded into the following equation: 
\begin{align}
\left\{
\begin{array}{l}
  \partial_t u^\varepsilon (x,t) =F\left(\frac{u^\varepsilon (x+\varepsilon,t-\varepsilon\tau_0(x))-u^\varepsilon (x,t-\varepsilon \tau_0(x))}{\varepsilon}\right) \hspace{1 cm} (x,t)\in \mathbb{R}\times (0,T),  \label{exo} \\
  u^\varepsilon (x,s)=u_0(x,s) \hspace{4.8 cm} (x,s) \in \mathbb{R} \times [-2\varepsilon \tau,0],
\end{array}
\right.
\end{align}
with the following assumptions on $F$, $\tau_0$ and $u_0$:
\begin{align}
\left\{
\begin{array}{l}
 F \mbox{ is a non-decreasing, bounded and } C_F-\mbox{Lipschitz continuous function on }\mathbb{R},  \\
 \overline{\tau_0(\mathbb{R})}= [\xi, \tau], \\
   u_0 \mbox{ is a globally } L-\mbox{Lipschitz continuous function on }\mathbb{R} \times [-2\tau,0] . \label{assumptions}
\end{array}
\right.
\end{align}
As the argument of $F$ in \eqref{exo} converges towards a space derivative for vanishing $\varepsilon$, it is natural to introduce the following dynamics in the form of a first-order Hamilton-Jacobi equation:
\begin{align}
\left\{
\begin{array}{l}
  \partial_t u^0 (x,t) =F(\partial_x u^0 (x,t)) \hspace{1 cm} (x,t)\in \mathbb{R}\times (0,T),  \label{exo2} \\
  u^0(x,0)=u_0(x,0) \hspace{2.6 cm} x \in \mathbb{R}.
\end{array}
\right.
\end{align}
The first goal of the article is to prove that $u^\varepsilon$ tends locally uniformly towards $u^0$: 
\begin{theorem}[Homogenization of the delay equation] \label{convergenceth}$\,$\\
For $\tau\in \bigg(0;\frac{1}{e C_F}\bigg)$ and under assumptions \eqref{assumptions}, the solution $u^\varepsilon$ to \eqref{exo} converges locally uniformly towards the (unique viscosity) solution of \eqref{exo2}.
\end{theorem}
The key idea of the convergence proof relies on a strict comparison principle. This one is stated for the microscopic model for $\varepsilon=1$:  
\begin{align}
\left\{
\begin{array}{l}
  \partial_t u (x,t) =F({u (x+1,t-\tau_0(x))-u (x,t- \tau_0(x))}) \hspace{1 cm} (x,t)\in \mathbb{R}\times (0,T),  \label{exo3} \\
  u (x,s)=u_0(x,s) \hspace{5 cm} (x,s) \in \mathbb{R} \times [-2 \tau;0].
\end{array}
\right.
\end{align}
In contrast with the classical cases where it is sufficient to keep the initial order of the solutions, here, because of the reaction time, those solutions (namely the vehicles in traffic flow modelling) must be sufficiently spaced out at initial times in the sense made precise in the following theorem: 
\begin{theorem}[Strict comparison principle] \label{compa}$\,$\\
Under the assumptions \eqref{assumptions}, let $v$ and $u$ be respectively a supersolution and a subsolution of \eqref{exo3}. Let us assume that there exist $\delta>0$, $t_0\in [0,T)$, $R\in[0,+\infty)$, $x_0\in\mathbb{R}$, and $\rho$ a positive, non decreasing function on $\mathbb{R}$ such that:
\begin{align}
v(x,t)\geq u(x,t) \hspace{1 cm} \mbox{ for  } R\leq |x-x_0|\leq R+1,\hspace{0.3cm} t\in [t_0-\tau,T) \label{hypborne}
\end{align}
\begin{align}
\delta\leq (v-u)(x,t-\tau ')\leq \rho(\tau ') (v-u)(x,t)   \hspace{0.8 cm} \mbox{ for  } \tau '\in [0,\tau],  (x,t) \in [x_0-R,x_0+R] \times [t_0-\tau,t_0], \label{eqini1}
\end{align}
where $\rho$ is such that: 
\begin{align}
1+C_F \rho(\tau)\int_0^{\tau '}\rho(s)ds < \rho(\tau ')   \hspace{1 cm} \tau ' \in [0,\tau]. \label{condrho}
\end{align}
Then we have:
\begin{align}
 0<\delta e^{-C_F\rho(\tau)(t-t_0)}\leq v(x,t)-u(x,t) \hspace{1 cm}  \mbox{ for  } (x,t) \in [x_0-R,x_0+R]\times [t_0,T). \label{finalres}
\end{align}
\end{theorem}
\begin{remark}
The theorem is still valid when $R=+\infty$ (substituting $[x_0-R,x_0+R]$ by $\mathbb{R}$), \eqref{hypborne} is defined on $\emptyset$.
\end{remark}
%
\begin{remark}
All the proofs involved in this paper remain the same when dealing with stochastic reaction times provided those are uniformly bounded from above and from below. 
\end{remark}
The homogenization proof is different from the one performed in \cite{forca} on generalized Frenkel-Kontorova models, concerning an infinite system of particles that interact with a finite number of neighbours and that are subject to a periodic potential (this system takes the form of non-linear ODEs). The convergence proof in \cite{forca} relies on the construction of hull functions in a periodic setting. The case where particles interact with an infinite number of other particles is covered in \cite{disloc}. Other homogenization results in traffic flow modelling without considering the reaction time can be found in \cite{salaz1} when adding a junction condition in the microscopic model and in \cite{salaz2} for a second order microscopic model.

The second goal of the article is to provide a ``counter-example to homogenization'' in the following sense: we will display an explicit model for which $u^\varepsilon$ does not converge locally uniformly towards $u^0$ in the special case where all the drivers have a common reaction time larger than the threshold: 
\begin{theorem}[A counter-example to homogenization] \label{counterexample}$\,$\\
In the special case $\tau_0\equiv\tau$, there exist a Lipschitz continuous $F$, an initial data, and $\tau>\frac{1}{e C_F}$ such that the solution $u^\varepsilon$ to \eqref{exo} does not converge locally uniformly towards the solution of \eqref{exo2}.
\end{theorem}
We derive such a counter-example by a small perturbation of a trivial but unstable case where the homogenization holds for all reaction times. This trivial case is a stationnary one where all the vehicles are initially equally spaced. This also highlights the fact that the initial dynamics is at least as important as the value of the reaction time from the homogenization perspective. 
\subsection{Organisation of the article}
To get the homogenization result, we will first show the existence and uniqueness of solutions to \eqref{exo} and \eqref{exo2} in Section \ref{exun}. We will then prove the strict comparison principle for \eqref{exo} in Section \ref{proofcompa} and we will show that the restriction on $\tau$ in Theorem \ref{convergenceth} is equivalent to the existence of functions $\rho$ that verify \eqref{condrho}. We will also state and prove a direct corollary of Theorem \ref{compa} on drivers' positions and give a counter-example to Theorem \ref{compa} when the vehicles are not suitably spaced out at initial times, regardless of the threshold on $\tau$.  We will give the explicit convergence proof in Section \ref{proofconv} and give the unstable example where homogenization holds for all reaction times, provided the vehicles are perfectly spaced at initial times. In section \ref{counter}, we will give the explicit model and will prove Theorem \ref{counterexample}.
\section{Existence and uniqueness } \label{exun}
This section is devoted to the study of the existence and uniqueness of solutions to \eqref{exo} and \eqref{exo2}.
Studying existence and uniqueness of solutions to \eqref{exo} is stricly equivalent to studying \eqref{exo3}.
\begin{proposition}
Under the assumptions \eqref{assumptions}, there exists a unique classical solution to \eqref{exo3}.\label{uniqeps1}
\end{proposition}
\begin{proof}[Proof of Proposition 2.1] 
\textcolor{white}{.} \\
The proof is based on an explicit and incremental construction. Indeed, on the time interval $(0;\xi]$ the equation can be written as follows:
$$\partial_t u (x,t) =F({u_0 (x+1,t-\tau_0(x))-u_0 (x,t-\tau_0(x))}) \hspace{2 cm} (x,t)\in \mathbb{R}\times (0,\xi].$$
Hence, it can be explicitly integrated and we get: 
$$ u(x,t)=u_0(x,0)+\int_0^t F({u_0 (x+1,s-\tau_0(x))-u_0 (x,s- \tau_0(x))}) ds \hspace{1 cm} \hspace{2 cm} (x,t)\in \mathbb{R}\times (0,\xi].$$
This enables to define the function $u_1$ which coincides with $u_0$ at initial times and that is defined by the previous expression in $[0,\xi]$.

Hence, in the time interval $(\xi;2\xi]$, the equation can be written as follows:
$$\partial_t u (x,t) =F({u_1 (x+1,t-\tau_0(x))-u_1 (x,t- \tau_0(x))}) \hspace{0.7 cm} (x,t)\in \mathbb{R}\times (\xi;2\xi].$$
Here again, it can be explicitly integrated: 
$$ u(x,t)=u_1(x,\xi)+\int_\xi^t F({u_1 (x+1,s-\tau_0(x))-u_1 (x,s- \tau_0(x))}) ds \hspace{1 cm} (x,t)\in \mathbb{R}\times [\xi,2\xi].$$
The process can be iterated to obtain a function $u$ defined piecewise which is, by construction, continuous on $\mathbb{R}\times [0,\infty)$ and which solves equation \eqref{exo3} in each $\mathbb{R}\times(k\xi,(k+1)\xi)$, $k= 0,.., \lfloor \frac{T}{\xi}\rfloor -1$. It remains to prove that it solves the equation globally and so, that $u$ is $C^1$ in time on $(0,T)$. Indeed, at each $k\xi$, the left and right limits of the derivative coincide and verify the equation.

The uniqueness also comes from the process as we see that on each time interval, the solution is completely determined by its expression on the previous interval which implies global uniqueness for a given initial condition $u_0$.
\end{proof}

Equation \eqref{exo2} is an Hamilton-Jacobi equation and is therefore studied in the viscosity solutions' framework. The existence and uniqueness are classical and are respectively given by Perron's method and the usual comparison principle, see \cite{barles} for instance.  
\section{Strict comparison principle} \label{proofcompa}
This section is first devoted to the proof of Theorem \ref{compa}. We will then explain the restriction on $\tau$. We will also prove the conservation of initial order of the vehicles as a corollary when those are well spaced out at initial times. Finally, we will finally present a counter-example to comparison principle for any reaction time $\tau>0$. Indeed, the existence of $\rho$ functions that verify \eqref{condrho} is equivalent to consider $\tau$ under the threshold but both \eqref{condrho} and \eqref{eqini1} are necessary for the strict comparison principle to hold and hence, the restriction on $\tau$ is not a sufficient condition to ensure the result. The existence of $\rho$ functions is not enough, it has to be linked with the initial dynamics of the vehicles. Namely, we will show that if the vehicles are not suitably spaced out at initial times (in the sense made precise in \eqref{eqini1}), their initial order can be disturbed even for reaction times that are below the threshold. 
\subsection{Proof of Theorem \ref{compa}}
\begin{proof}[Proof of Theorem 1.2]
\textcolor{white}{.} \\
Let us consider $d:=v-u$ and define $T^*$ as: $$T^*=\sup \{S\in [0,+\infty)/\forall \tau'\in[0,\tau], \forall  (x,t) \in [x_0-R,x_0+R] \times [t_0-\tau,S], \hspace{0.2 cm} d(x,t-\tau')\leq \rho (\tau') d(x,t)  \}.$$ 
The set is not empty as it contains $t_0$ so $T^*$ is well-defined. 

%
\begin{itemize}
\item We first claim that to establish \eqref{finalres}, it is sufficient to prove that $T^* \geq T$.
\end{itemize}
Thanks to \eqref{condrho} we have $\rho(0)>1$. Taking $\tau '=0$ in the definition of $T^*$, with $T^*\geq T$, implies that:
\begin{align}
d(x,t)\geq 0 \hspace{1cm}(x,t)\in [x_0-R,x_0+R]\times [t_0-\tau,T). \label{posi}
\end{align}
By combining with \eqref{hypborne}, we get: 
\begin{align}
d(x,t)\geq 0 \hspace{1cm}(x,t)\in [x_0-R-1,x_0+R+1]\times [t_0-\tau,T). \label{posi2}
\end{align}
%
Let us consider $(x,t)\in [x_0-R,x_0+R]\times (t_0,T)$. By definition of $d$, we have: 
\begin{align}
\partial_t d(x,t) \geq F({v (x+1,t-\tau_0(x))-v (x,t- \tau_0(x))})- F({u (x+1,t-\tau_0(x))-u (x,t- \tau_0(x))}). \label{lipdessous}
\end{align}
By combining \eqref{posi2} with the fact that $F$ is non-decreasing gives: 
\begin{align*}
\partial_t d(x,t) \geq F({u (x+1,t-\tau_0(x))-v (x,t- \tau_0(x))})- F({u (x+1,t-\tau_0(x))-u (x,t- \tau_0(x))}).
\end{align*}
Using the fact that $F$ is $C_F-$Lipschitz, we get:  
\begin{align*}
\partial_t d(x,t) \geq -C_F d(x,t-\tau_0(x)).
\end{align*}
If $T^* \geq T$, then we have:
\begin{align*}
\partial_t d(x,t) \geq -C_F\rho(\tau_0(x)) d(x,t).
\end{align*}
By using the fact that $\rho$ is non-decreasing, we get: 
\begin{align}
\partial_t d(x,t) \geq -C_F\rho(\tau) d(x,t). \label{decroi}
\end{align}
Hence, $d$ is a supersolution of the problem: 
\begin{align}
\left\{
\begin{array}{l}
  \partial_t w(x,t)=-C_F \rho(\tau)w(x,t), \\
  w(x,t_0)=\delta. \label{sys2}
\end{array}
\right.
\end{align}
We conclude \eqref{finalres} by a standard comparison principle for this ODE. 

\begin{itemize}
\item Let us now show that $T^*\geq T$. 
\end{itemize}
By contradiction, if we suppose that $T^* < T$, then for all $\beta\in \left(0,\inf\left(\frac{T-T^*}{2},1 \right) \right)$, there exists $\tau_\beta\in [0,\tau]$, $(x_\beta, t_\beta)\in [x_0-R,x_0+R]\times (T^*,T^*+\beta]$ such that: 
\begin{align}
d(x_\beta, t_\beta-\tau_\beta)> \rho (\tau_\beta) d(x_\beta, t_\beta). \label{absurde}
\end{align}
Let us notice that \eqref{decroi} holds true for all $t\in [t_0,T^*]$ (and so does \eqref{finalres}). 

As a preliminary result, let us first show that: 
\begin{align}
d(x,t-\tau')\leq \bar{\rho}(\tau')d(x,t) \hspace{1 cm} \tau ' \in [0,\tau],  (x,t) \in [x_0-R,x_0+R] \times [t_0,T^*], \label{controstrict}
\end{align}
with $\bar{\rho}$ given by: $$\bar{\rho}( \tau')= 1+C_F \rho(\tau)\int_0^{\tau '}\rho(s)ds. $$ By remarking that $\bar{\rho}(0)=1$, we notice the equality for $\tau'=0$. For $\tau'\in [0,\tau], (x,t)\in [x_0-R,x_0+R] \times [t_0,T^*]$, we have: 
$$ \partial_{\tau'}(\bar{\rho}(\tau')d(x,t))=C_F\rho(\tau)\rho(\tau')d(x,t).$$ 
Let us remark that $\tau' \mapsto d(x,t-\tau')$ is a subsolution of this equation. Indeed:
\begin{align}
\partial_{\tau'} \left(d(x,t-\tau')\right)= -\partial_t d(x,t-\tau').
\end{align}
We can use \eqref{decroi} as $t-\tau'\in [t_0-\tau,T^*]$: 
\begin{align}
\partial_{\tau'} \left(d(x,t-\tau')\right) \leq C_F \rho(\tau) d(x,t-\tau ').
\end{align}
Finally, by definition of $T^*$, we get: 
\begin{align}
\partial_{\tau'} \left(d(x,t-\tau')\right) \leq C_F \rho(\tau) \rho(\tau')d(x,t).
\end{align}
The preliminary result is then obtained by a comparison principle on the equation $\partial_{\tau'} W=C_F\rho(\tau)\rho(\tau')W$.
%

Let us notice that \eqref{assumptions} and \eqref{lipdessous} imply that: 
\begin{align}
\partial_t d(x,t)\geq -2 ||F||_\infty.
\end{align}
By integrating it, we get: 
\begin{align}
d(x,t)\geq d(x,s)-2||F||_\infty(t-s) \hspace{1cm}  (x,t) \in [x_0-R,x_0+R]\times [t_0,T],  s\in [t_0,t].
\label{lipdessous2}
\end{align}
\textbf{Case 1: $t_\beta-\tau_\beta\leq T^*$}.

In this case, by defining $\tau'_\beta:=T^*-t_\beta+\tau_\beta\in [0,\tau]$ we get: $$d(x_\beta, t_\beta-\tau_\beta)=d(x_\beta, T^*-\tau'_\beta).$$
Hence, by \eqref{controstrict} we get: 
$$d(x_\beta, t_\beta-\tau_\beta)\leq \bar{\rho}(\tau'_\beta)d(x_\beta,T^*). $$
Thanks to \eqref{lipdessous2}, we have: 
\begin{align}
d(x_\beta, t_\beta-\tau_\beta)\leq \bar{\rho}(\tau'_\beta)(d(x_\beta,t_\beta)+2||F||_\infty \beta). \label{intermediaire}
\end{align}
In order to get a contradiction with \eqref{absurde}, it is sufficient to consider $\beta<\min (\beta_1,\beta_2)$ with:
\begin{align}
\left\{
\begin{array}{l}
  \beta_1:=\frac{\delta}{2||F||_\infty+1} e^{-C_F\rho(\tau)(T^*-t_0)}>0, \\
  \beta_2:= \frac{\beta_1}{2||F||_\infty} \inf_{[0,\tau]} (\frac{\rho}{\bar{\rho}}-1)>0.
\end{array}
\right.
\end{align}

Indeed, for $\beta<\beta_1$, thanks to \eqref{finalres} (which is still valid in $[x_0-R,x_0+R]\times [t_0,T^*]$) and to \eqref{lipdessous2}, we have:
\begin{align*}
d(x_\beta, t_\beta)\geq d(x_\beta,T^*)-2||F||_\infty\beta\geq \delta e^{-C_F\rho(\tau)(T^*-t_0)}-2||F||_\infty\beta_1=\beta_1,   
\end{align*}
and for $\beta<\min (\beta_1,\beta_2)$ we have: 
\begin{align*}
\bar{\rho}(\tau'_\beta)(d(x_\beta,t_\beta)+2||F||_\infty\beta)\leq \rho (\tau'_\beta) d(x_\beta,t_\beta).
\end{align*}
Together with \eqref{intermediaire}, this implies: 
$$d(x_\beta, t_\beta-\tau_\beta)\leq \rho (\tau'_\beta) d(x_\beta,t_\beta). $$
Finally, as $\rho$ is non decreasing and $\tau'_\beta\leq \tau_\beta$, we get a contradiction with \eqref{absurde}.

\textbf{Case 2: $t_\beta-\tau_\beta> T^*$}.

As we have $t_\beta<T^*+\beta$, this implies that $\tau_\beta\in[0,\beta]$.

Then, thanks to \eqref{lipdessous2} and \eqref{absurde} we have:
\begin{align*}
d(x_\beta,t_\beta)+2||F||_\infty\beta\geq d(x_\beta,t_\beta)+ 2||F||_\infty\tau_\beta\geq d(x_\beta, t_\beta-\tau_\beta)>\rho (\tau_\beta) d(x_\beta,t_\beta).
\end{align*}
In particular, we have: 
\begin{align*}
(\rho(\tau_\beta)-1)d(x_\beta,t_\beta)<2||F||_\infty\beta.
\end{align*}
Thanks to \eqref{lipdessous2}, this implies: 
\begin{align*}
(\rho(\tau_\beta)-1)d(x_\beta,T^*)<2||F||_\infty\rho(\tau_\beta)\beta.
\end{align*}
Thus,
\begin{align}
0< \delta e^{-C_F\rho(\tau)(T^*-t_0)} <\frac{\rho(\tau_\beta)}{\rho(\tau_\beta)-1}2||F||_\infty\beta.
\end{align}
For vanishing $\beta$ this implies $\delta e^{-C_F\rho(\tau)(T^*-t_0)}=0 $ which is false. 
\end{proof}
\subsection{Restriction on $\tau$ and existence of solutions to \eqref{condrho}}
Condition \eqref{condrho} explains why we chose $\tau<\frac{1}{eC_F}$. In fact, we have the following proposition: 
\begin{proposition}\label{A}
\eqref{condrho} admits solutions if and only if $\tau<\frac{1}{eC_F}$. 
\end{proposition}
We now give a proof of Proposition \ref{A}.
\begin{proof}[Proof of Proposition 3.1]
We want to show that condition \eqref{condrho} implies $\tau<\frac{1}{eC_F}$.

Let us define: $$ R(\tau ')=\int_0^{\tau '} \rho(s)ds. $$
We have $R(0)=0$. By defining $\alpha=R'(\tau)=\rho(\tau)$, condition \eqref{condrho} is equivalent for $R$ to be a strict supersolution of: 
\begin{align}
\left\{
\begin{array}{l}
  y'=C_F\alpha y+1,  \mbox {   in } [0,\tau], \\
  y(0)=0. \label{sys}
\end{array}
\right.
\end{align}
As $S: \tau ' \mapsto \frac{e^{C_F\alpha\tau '}-1}{C_F\alpha}$ is a solution of this Cauchy problem on $[0,\tau]$, by Gronwall's lemma or a standard comparison principle, we deduce that $R\geq S$ or equivalently that: $$1+C_F R(\tau ')R'(\tau)\geq e^{C_F R'(\tau)\tau '} \hspace{1 cm} \tau ' \in [0,\tau]. $$

Moreover, as $R$ is a strict supersolution of \eqref{sys}, we deduce that: $R'(\tau ') > e^{C_F R'(\tau)\tau '}$ for all $\tau ' \in [0,\tau]$. By considering this inequality for $\tau '= \tau$ and taking the logarithm, we get: $$C_F \tau < \frac{ \ln R'(\tau)}{R'(\tau)}\leq \max_{x>0} \frac{\ln x}{x}=\frac{1}{e}. $$

Let us now show that \eqref{condrho} admits solutions as soon as $\tau<\frac{1}{e C_F}$.

It is equivalent to find positive and non-decreasing solutions $R$ of:
\begin{align}
\left\{
\begin{array}{l}
  R'(\tau')>C_F R'(\tau)R(\tau')+1 \hspace{2cm} \tau'\in[0,\tau], \\
  R(0)=0. \label{sys2}
\end{array}
\right.
\end{align}
Let us consider $\lambda\in\left(1,\sqrt{\frac{1}{e\tau C_F}}\right)$. 

The function $h:\gamma \mapsto \ln \gamma - \lambda \gamma \tau C_F$ is smooth on $(0,+\infty)$ and admits a maximum at $\gamma_{max}=\frac{1}{\lambda\tau C_F}$. As $\lambda<\sqrt{\frac{1}{e\tau C_F}}$, it is straightforward to check that $\ln \lambda < h(\gamma_{max})$ and so, there exists $\gamma_0>0$ such that $h(\gamma_0)=\ln \lambda$ or equivalently that $\gamma_0=\lambda e^{\lambda\gamma_0 C_F \tau}$.

Let us define $U_\lambda: \tau' \mapsto \frac{1}{\gamma_0 C_F} (e^{\lambda\gamma_0 C_F \tau'}-1).$
We have: $U_\lambda'(\tau)=\lambda e^{\lambda\gamma_0 C_F \tau}=\gamma_0$.
As $\lambda>1$, $U_\lambda$ verifies:
\begin{align*}
\left\{
\begin{array}{l}
  U_\lambda '(\tau')=\lambda \gamma_0 C_F U_\lambda(\tau')+\lambda>\gamma_0 C_F U_\lambda(\tau')+1=C_F U_\lambda'(\tau) U_\lambda(\tau')+1\hspace{2cm} \tau'\in[0,\tau], \\
  U_\lambda(0)=0.
\end{array}
\right.
\end{align*}
For all $\lambda\in\left(1,\sqrt{\frac{1}{e\tau C_F}}\right)$, we can construct a solution $U_\lambda$ to \eqref{sys2}. Therefore, as soon as $\tau<\frac{1}{e C_F}$, there exist infinitely many solutions. 
\end{proof}
\begin{remark}
Condition \eqref{condrho} is reduced to: $$\rho C_F \tau ' <1-\frac{1}{\rho },$$ if we consider $\rho$ as a constant. If it is considered for $\tau'=\tau$, we get $ \tau<\frac{1}{4C_F}$ because the condition tells that the function $x\mapsto C_F \tau x^2-x+1$ has a negative part. Conversely, for all fixed $\tau<\frac{1}{4 C_F}$, we can always find constant solutions. 
\end{remark}
\subsection{Consequence on drivers' positions and a counter-example}
A very practical consequence of this strict comparison principle for the microscopic model derived from traffic flow is the conservation of initial order for vehicles, provided they are suitably spaced out at initial times. 
\begin{corollary} [Conservation of initial order]
Let us consider $(X_i)_{i\in \mathbb{Z}}$ a sequence of drivers' positions that evolve under the dynamics \eqref{microdelay} with initial conditions given by \eqref{microinitial} and for $\tau_0\equiv\tau$. We consider the solution $u$ to \eqref{exo3} such that: $$u_0(i,t)=x_i^0(t) \hspace{0.5cm} (i,t)\in \mathbb{Z}\times [-\tau,0].$$ We suppose that there exist $\sigma>0$, a time function $\rho_{p}$ which verifies \eqref{condrho} such that:
\begin{align}
\sigma \leq u_0(x+1,t-\tau')-u_0(x,t-\tau') \leq \rho_{p}(\tau') (u_0(x+1,t)-u_0(x,t)),\mbox{   } \tau'\in[0,\tau], (x,t)\in \mathbb{R}\times [-\tau,0]. \label{goodhyp}
\end{align}
Then, 
$$X_{i+1}(t) > X_{i}(t)    \hspace{0.5cm} (i,t)\in \mathbb{Z}\times [0,T). $$
\end{corollary}
\begin{proof}[Proof of Corollary 3.3]
\textcolor{white}{.} \\
By uniqueness, we have:
\begin{align}
X_i(t)=u(i,t) \hspace{0.5cm} (i,t)\in \mathbb{Z}\times [0,T). \label{edoedp}
\end{align}
We consider $v:(x,t)\mapsto u(x+1,t)$. In particular, we have: 
\begin{align}
X_{i+1}(t)=v(i,t) \hspace{0.5cm} (i,t)\in \mathbb{Z}\times [0,T). \label{edoedp2}
\end{align}
For $d:=v-u$, thanks to \eqref{goodhyp}, the conditions of Theorem \ref{compa} (and its remark) are fulfilled with $\delta = \sigma$, $t_0=0$, $R=+\infty$ and $\rho=\rho_{p}$. Therefore we have $v>u$ on $\mathbb{R}\times [0,T)$ and in particular on $\mathbb{Z}\times [0,T)$. Using \eqref{edoedp} and \eqref{edoedp2} this is equivalent to have: 
$$X_{i+1}(t) > X_{i}(t)    \hspace{2cm} (i,t)\in \mathbb{Z}\times [0,T). $$
\end{proof}
\begin{remark}
Our microscopic model is valid when the drivers always stay in the same order, otherwise it is not adapted anymore to the physical situation. More explicitly, the behaviour of the driver at position $X_i$ does not depend anymore on the position $X_{i+1}$ if at some time the car $i+1$ is not in front of the car $i$.
\end{remark}
In contrast with most Hamilton-Jacobi equations, it is not possible to state a classical comparison principle because of the delay time. More accurately, it is compulsory to ensure that the vehicles are suitably spaced out (this space is represented by the function $\rho$).

\begin{example} \label{nonconserv}
We consider the case: $\tau_0\equiv\tau$. For any delay time $\tau>0$, it is straightforward to show that the initial order of solutions is sometimes not conserved when the vehicles are not suitably spaced in the sense made precise in Theorem \ref{compa}. Let us consider $n_0\in\mathbb{N}^*$ such that $\tau>\frac{2}{n_0}$. Let us consider the particular case for $F$ being identically equal to the identity on $[0;1]$ (and that verifies \eqref{assumptions} in $\mathbb{R}$). Let us consider two sets of drivers $(X_i)_{i\in \mathbb{Z}}$ and $(Y_i)_{i\in \mathbb{Z}}$ that evolve under the same dynamics \eqref{microdelay} and such that: $$y_i^0(t)< x_i^0(t) \hspace{2cm}  (i,t)\in \mathbb{Z}\times [-\tau,0], $$ and there exists $j\in\mathbb{Z}$ such that for $t\in [-\tau,0]$: 
\begin{align}
\left\{
\begin{array}{l}
 y_j^0(t)=j-1+\frac{n_0}{n_0+1} e^{n_0 t}< x_j^0(t)=j,  \\
 y_{j+1}^0(t)=j+\frac{n_0}{n_0+1}e^{n_0 t}, \\
 x_{j+1}^0(t)=y_{j+1}^0(t) +\frac{1}{n_0+1}
\end{array}
\right.
\end{align}
Those two sets of drivers do not see each other, that means that they do not evolve on the same physical road but on two identical copies of the real line. We can integrate \eqref{microdelay} on $[0,\tau]$ for the two sets and we find that: $$(Y_j-X_j)(\tau)=\frac{n_0}{n_0+1}\tau-\frac{1}{n_0+1}-\int_{-\tau}^0 e^{n_0 u} du=\frac{n_0}{n_0+1}\tau-\frac{1}{n_0+1}-\frac{1-e^{-n_0\tau}}{n_0+1}>0. $$
Therefore, the initial order is disrupted even if: $$\frac{1}{n_0+1}\leq x_i^0(t)- y_i^0(t)\hspace{2cm}  (i,t)\in \mathbb{Z}\times [-\tau,0].$$ 
\end{example}

\section{Convergence} \label{proofconv}

This section is mainly devoted to the proof of Theorem \ref{convergenceth}. The unstable stationnary case will be given as the final remark of this section. 

\subsection{Proof of Theorem \ref{convergenceth}}
\begin{proof}[Proof of Theorem 1.1]
\textcolor{white}{.} \\
Let us first show that $u^\varepsilon$ is globally Lipschitz continuous in space and time uniformly in $\varepsilon$. 

In time: By looking at \eqref{exo}, we remark that $u^\varepsilon$ $||F||_\infty-$Lipschitz continuous in time.

%
%


In space: The equation is translation invariant in space and invariant by addition of constants to the solutions. Let $h>0$. 

The solution corresponding to the initial condition $u_1: (x,t)\mapsto u_0(x+h,t)$ is the function $w: (x,t)\mapsto u^\varepsilon(x+h,t)$.

The one associated to $u_2: (x,t)\mapsto u_0(x,t)+2 L h$ is $v: (x,t)\mapsto u^\varepsilon(x,t)+2 L h$.
%

We define $\psi_1: (x,t)\mapsto \frac{1}{\varepsilon} w(\varepsilon x, \varepsilon t)$ and $\psi_2: (x,t)\mapsto \frac{1}{\varepsilon} v(\varepsilon x, \varepsilon t) $. 

$\psi_1$ solves: 
\begin{align}
\left\{
\begin{array}{l}
  \partial_t u (x,t) =F({u (x+1,t-\tau_0(\varepsilon x))-u (x,t- \tau_0(\varepsilon x))}) \hspace{1 cm} (x,t)\in \mathbb{R}\times (0,T),  \label{psi1} \\
  u (x,t)=\frac{1}{\varepsilon} u_0(\varepsilon x+h,\varepsilon t) \hspace{4 cm} (x,t) \in \mathbb{R} \times [-2 \tau;0].
\end{array}
\right.
\end{align}
$\psi_2$ solves: 
\begin{align}
\left\{
\begin{array}{l}
  \partial_t u (x,t) =F({u (x+1,t-\tau_0(\varepsilon x))-u (x,t- \tau_0(\varepsilon x))}) \hspace{1 cm} (x,t)\in \mathbb{R}\times (0,T),  \label{psi2} \\
  u (x,t)=\frac{1}{\varepsilon} u_0(\varepsilon x,\varepsilon t)+\frac{2Lh}{\varepsilon} \hspace{3.6 cm} (x,t) \in \mathbb{R} \times [-2 \tau;0].
\end{array}
\right.
\end{align}
By \eqref{assumptions}, $\psi_1$ and $\psi_2$ are $L-$ Lipschitz continuous functions on $\mathbb{R}\times [-2\tau,0]$. 

Hence, $d:=\psi_2-\psi_1$ is a $2L-$ Lipschitz continuous function on $\mathbb{R}\times [-2\tau,0]$. Otherwise, we have: 
\begin{align}
d(x,t)= \frac{1}{\varepsilon} \left( u_0(\varepsilon x,\varepsilon t)-u_0(\varepsilon x+h,\varepsilon t)\right)+\frac{2Lh}{\varepsilon}\geq -\frac{Lh}{\varepsilon}+\frac{2Lh}{\varepsilon}=\frac{Lh}{\varepsilon} \hspace{0.5cm} (x,t) \in \mathbb{R} \times [- \tau;0]\label{dgood}.
\end{align}
Let us now consider $\rho_m$ any solution to \eqref{condrho} for $\tau<\frac{1}{eC_F}$ (see Proposition \ref{A}). We recall that \eqref{condrho} implies $\rho_m\geq \rho_m(0)>1$.

We have: $$d(x,t-\tau')\leq d(x,t)+2L\tau \hspace{0.5cm}\tau'\in[0,\tau], (x,t)\in \mathbb{R}\times [-\tau,0].$$
Therefore, to get: $$d(x,t-\tau')\leq \rho_m(\tau')d(x,t) \hspace{0.5cm}\tau'\in[0,\tau], (x,t)\in \mathbb{R}\times [-\tau,0],$$ 
it is sufficient to verify: 
$$2L\tau\leq (\rho_m(\tau')-1)d(x,t) \hspace{0.5cm}\tau'\in[0,\tau], (x,t)\in \mathbb{R}\times [-\tau,0].$$
As $\rho_m$ is a non-decreasing function, and as $d\geq 0$, it is finally sufficient to verify: 
$$\frac{2L\tau}{\rho_m(0)-1}\leq d(x,t) \hspace{0.5cm}\tau'\in[0,\tau], (x,t)\in \mathbb{R}\times [-\tau,0].$$
Thanks to \eqref{dgood}, this will be the case for $\varepsilon$ small enough. 
The conditions of the strict comparison principle are fulfilled with $\delta = \frac{Lh}{\varepsilon}$, $t_0=0$, $R=+\infty$ and $\rho=\rho_m$. Then, by Theorem \ref{compa} we have $\psi_2-\psi_1>0$ in $\mathbb{R}\times [0,T)$ and this implies for all $(y,s)\in \mathbb{R}\times [0,T)$: $$u^\varepsilon(y+h,s)\leq u^\varepsilon(y,s)+2 L h. $$ 
Analogously, we can show that: $$u^\varepsilon(y+h,s)\geq u^\varepsilon(y,s)-2 L h. $$

Hence, $u^\varepsilon$ is $2L$-Lipschitz continuous in space. Therefore, we can define the relaxed upper and lower semi-limits  $\bar u(x,t)=\lim_{\varepsilon \rightarrow 0} \sup^* u^\varepsilon(x,t)$ and $ \underline {u}(x,t)=\lim_{\varepsilon \rightarrow 0} \inf_* u^\varepsilon(x,t)$. By definition, we have $\underline {u} \leq \bar u$.

Given $\nu>0$, for $\varepsilon$ small enough, we can apply Theorem \ref{compa} on $v_b: (x,t)\mapsto \frac{1}{\varepsilon} \left(u_0(\varepsilon x,\varepsilon t)+\nu+(||F||_\infty+L) \varepsilon t \right)$ 
 and on $u_b: (x,t)\mapsto \frac{1}{\varepsilon} u^\varepsilon (\varepsilon x,\varepsilon t) $ and we get: $$u^\varepsilon(x,t) \leq u_0(x,t)+\nu+(||F||_\infty+L) t \qquad (x,t)\in \mathbb{R}\times [-2\varepsilon \tau,T), $$ 
which gives for vanishing $\nu$: $$u^\varepsilon(x,t) \leq u_0(x,t)+(||F||_\infty+L) t \qquad (x,t)\in \mathbb{R}\times [-2\varepsilon \tau,T). $$
Similarly, we get: $$u^\varepsilon(x,t) \geq u_0(x,t)-(||F||_\infty+L) t \qquad (x,t)\in \mathbb{R}\times [-2\varepsilon \tau,T), $$
and so: $$|u^\varepsilon(x,t) -u_0(x,t)|\leq (||F||_\infty+L) t \qquad (x,t)\in \mathbb{R}\times [-2\varepsilon \tau,T). $$
This implies: $$\bar u(0,x)= \underline{u}(0,x)=u_0(x,0) \qquad x\in\mathbb{R}.$$
To get the convergence, it is sufficient to show that $\bar u$ and $\underline{u}$ are respectively a subsolution and a supersolution of \eqref{exo2}. Indeed, by using a classical comparison principle we will then get: 
$$\underline{u}\leq \bar{u}\leq u^0\leq\underline {u}$$
Let us only show by contradiction that $\bar u$ is a subsolution of \eqref{exo2} (the other one being very similar). Assume that there exist $(\bar x, \bar t)$, $\varphi\in C^1_{x,t}$, $(r,\theta, \eta)\in (0,+\infty)^3$ such that: 

\begin{align}
\left\{ \label{testfun}
\begin{array}{l}
  \bar u(\bar x, \bar t)=\varphi (\bar x, \bar t), \\
  \bar u < \varphi \mbox{ in } B_{2r}(\bar x, \bar t)\setminus\{(\bar x, \bar t)\}, \\
  \bar u \leq \varphi -2 \eta \mbox{ in } B_{2r}(\bar x, \bar t)\setminus B_{r}(\bar x, \bar t),\\
  \partial_t\varphi (\bar x, \bar t)= 2 \theta + F(\partial_x \varphi (\bar x, \bar t)), \\
  \partial_t\varphi ( x, t)\geq \theta + F(\partial_x \varphi (x,t)) \mbox{ in } B_{r}(\bar x, \bar t),
\end{array}
\right.
\end{align}
where we set $B_s(p,q):=(p-s,p+s)\times (q-s,q+s)$. 

As $u^\varepsilon$ and $\varphi$ are continuous, we can define:  
\begin{align}
M_\varepsilon:= \max_{B_{2r}(\bar x, \bar t)} (u^\varepsilon- \varphi):=(u^\varepsilon- \varphi)(x_\varepsilon, t_\varepsilon).
\end{align}

$M_\varepsilon\geq -\eta$ for $\varepsilon$ small enough and hence $(x_\varepsilon, t_\varepsilon)\in B_{r}(\bar x, \bar t)$. 


Let us define: $\varphi^\varepsilon:=\varphi+M_\varepsilon$. This function satisfies: 
\begin{align}
\left\{ \label{testfun2}
\begin{array}{l}
  \varphi^\varepsilon(x_\varepsilon, t_\varepsilon)=u^\varepsilon(x_\varepsilon, t_\varepsilon)\\
  \partial_t\varphi^\varepsilon ( x, t)\geq \theta + F(\partial_x \varphi^\varepsilon (x,t)) \mbox{ in } B_{r}(\bar x, \bar t).
\end{array}
\right.
\end{align}
By regularity of $\varphi^\varepsilon$ and $F$, for $\varepsilon$ small enough, we have in $B_{r}(\bar x, \bar t) $: 
\begin{align}
 \partial_t \varphi^\varepsilon(x,t) \geq F\bigg(\frac{\varphi^\varepsilon (x+\varepsilon,t-\varepsilon\tau_0( x))-\varphi^\varepsilon (x,t-\varepsilon \tau_0( x))}{\varepsilon}\bigg)+\frac{\theta}{2}.
\end{align}
Let us define: $d^\varepsilon(x,t)=\varphi^\varepsilon ( x, t)-u^\varepsilon ( x, t)$. In $B_{2r}(\bar x, \bar t)\setminus B_{r}(\bar x, \bar t)$ we have: 
\begin{align}
d^\varepsilon(x,t)=\varphi ( x, t)-u^\varepsilon ( x, t)+M_\varepsilon\geq \frac{3\eta}{2} - \eta= \frac{\eta}{2}. \label{minordeps}
\end{align}
Using the fact that $u^\varepsilon$ is $||F||_\infty-$Lipschitz continuous in time and that $\varphi$ is smooth, we remark that $d^\varepsilon$ is Lipschitz continuous in $B_{2r}(\bar x, \bar t)$. Let $K$ denote its Lipschitz constant. 
Let us define: $$d_1: (x,t)\mapsto \frac{1}{\varepsilon} d^\varepsilon (\varepsilon x, \varepsilon t).$$
$d_1$ is $K-$ Lipschitz continuous in $B_{\frac{2r}{\varepsilon}}\left(\frac{\bar x}{\varepsilon}, \frac{\bar t}{\varepsilon}\right)$. Thanks to \eqref{minordeps}, we have in $B_{\frac{2r}{\varepsilon}}\left(\frac{\bar x}{\varepsilon}, \frac{\bar t}{\varepsilon}\right)\setminus B_{\frac{r}{\varepsilon}}\left(\frac{\bar x}{\varepsilon}, \frac{\bar t}{\varepsilon}\right)$:
\begin{align}
d_1(x,t)\geq \frac{\eta}{2\varepsilon}\label{d1good}.
\end{align}
%
Let us now consider $\rho_1$ any solution to \eqref{condrho} for $\tau<\frac{1}{eC_F}$ (see Proposition \ref{A}). We recall that \eqref{condrho} implies $\rho_1\geq \rho_1(0)>1$.

We have: $$d_1(x,t-\tau')\leq d_1(x,t)+K\tau \hspace{0.5cm}\tau'\in[0,\tau], (x,t)\in B_{\frac{2r}{\varepsilon}}\left(\frac{\bar x}{\varepsilon}, \frac{\bar t}{\varepsilon}\right)\setminus B_{\frac{r}{\varepsilon}}\left(\frac{\bar x}{\varepsilon}, \frac{\bar t}{\varepsilon}\right).$$
Therefore, to get: $$d_1(x,t-\tau')\leq \rho_1(\tau')d_1(x,t) \hspace{0.5cm}\tau'\in[0,\tau], (x,t)\in B_{\frac{2r}{\varepsilon}}\left(\frac{\bar x}{\varepsilon}, \frac{\bar t}{\varepsilon}\right)\setminus B_{\frac{r}{\varepsilon}}\left(\frac{\bar x}{\varepsilon}, \frac{\bar t}{\varepsilon}\right),$$ 
it is sufficient to verify: 
$$K\tau\leq (\rho_1(\tau')-1)d_1(x,t) \hspace{0.5cm}\tau'\in[0,\tau], (x,t)\in B_{\frac{2r}{\varepsilon}}\left(\frac{\bar x}{\varepsilon}, \frac{\bar t}{\varepsilon}\right)\setminus B_{\frac{r}{\varepsilon}}\left(\frac{\bar x}{\varepsilon}, \frac{\bar t}{\varepsilon}\right).$$
As $\rho_1$ is a non-decreasing function, and as $d_1\geq 0$, it is finally sufficient to verify: 
$$\frac{K\tau}{\rho_1(0)-1}\leq d_1(x,t) \hspace{0.5cm}\tau'\in[0,\tau], (x,t)\in B_{\frac{2r}{\varepsilon}}\left(\frac{\bar x}{\varepsilon}, \frac{\bar t}{\varepsilon}\right)\setminus B_{\frac{r}{\varepsilon}}\left(\frac{\bar x}{\varepsilon}, \frac{\bar t}{\varepsilon}\right).$$
Thanks to \eqref{d1good}, this will be the case for $\varepsilon$ small enough.
We choose $\varepsilon$ such that $\max(1,\tau)<\frac{r}{8\varepsilon}$. The previous inequalities enable us to apply the strict comparison principle with $\delta=\frac{\eta}{2\varepsilon}$, $R=\frac{3r}{2\varepsilon}$, $x_0=\frac{\bar{x}}{\varepsilon}$, $T=\frac{\bar{t}}{\varepsilon}+R$ and $t_0=\frac{\bar{t}}{\varepsilon}-R$. Indeed, we have: 
\begin{align}
\left\{ \label{inclu}
\begin{array}{l}
[x_0-R- 1 , x_0+R+1 ] \setminus ]x_0-R,x_0+R[\times [t_0-\tau,T)\subset B_{\frac{2r}{\varepsilon}}\left(\frac{\bar x}{\varepsilon}, \frac{\bar t}{\varepsilon}\right)\setminus B_{\frac{r}{\varepsilon}}\left(\frac{\bar x}{\varepsilon}, \frac{\bar t}{\varepsilon}\right) \\
 \mbox {}[x_0-R,x_0+R] \times [t_0-\tau,t_0] \subset B_{\frac{2r}{\varepsilon}}\left(\frac{\bar x}{\varepsilon}, \frac{\bar t}{\varepsilon}\right)\setminus B_{\frac{r}{\varepsilon}}\left(\frac{\bar x}{\varepsilon}, \frac{\bar t}{\varepsilon}\right).
\end{array}
\right.
\end{align}
Hence, by Theorem \ref{compa}, we have $d_1>0$ in $[x_0-R,x_0+R]\times [t_0,T)$ which is in contradiction with the definition of $(x_\varepsilon, t_\varepsilon)$ as we have: 
\begin{align*}
\left\{
\begin{array}{l}
\left(\frac{x_\varepsilon}{\varepsilon}, \frac{t_\varepsilon}{\varepsilon}\right)\in B_{\frac{r}{\varepsilon}}\left(\frac{\bar x}{\varepsilon}, \frac{\bar t}{\varepsilon}\right)\subset [x_0-R,x_0+R]\times [t_0,T) , \\
 d_1\left(\frac{x_\varepsilon}{\varepsilon}, \frac{t_\varepsilon}{\varepsilon}\right)=0.
\end{array}
\right.
\end{align*}
\end{proof}
\subsection{A special case: homogenization for any reaction time}
A natural question arises. What happens for higher reaction times? Example \ref{nonconserv} highlights the fact that the initial dynamics is very important and that reaction times cannot be considered separately. The answer to this question is not invariable and the following example shows that the expected macrosopic model can be derived for any reaction time for a special initial condition. 
\begin{example} \label{unstable}
For $T=+\infty$, let us consider the case where all drivers have the same reaction time $\tau\in (0,+\infty)$. Let $L>0$ be the common interdistance between all the vehicles. We consider that the vehicles do not move at initial times:
$$x_i^0(t)=L i, \qquad t\in[-\tau,0].$$

This corresponds to: $$u_0(x,t)=L x, \qquad (x,t)\in\mathbb{R}\times [-\tau,0].$$

By incremental construction (or directly by uniqueness), we see that the solution $u^\varepsilon$ to \eqref{exo} for this initial condition does not depend on $\varepsilon$ and is given by: $$u^\varepsilon(x,t)=Lx+F(L)t, \qquad (x,t)\in\mathbb{R}\times [0,+\infty).$$

The unique solution $u^0$ of \eqref{exo2} corresponding to this initial linear data is also given by the same expression. Therefore we have: $$u^\varepsilon(x,t)=u^0(x,t)=Lx+F(L)t, \qquad (x,t)\in\mathbb{R}\times [0,+\infty).$$

\end{example}
%
\section{A counter-example to homogenization} \label{counter}
The goal of this section is to exhibit a counter-example derived from Example \ref{unstable}. We will first give the explicit expressions of $F$, of the vehicles' initial positions and the value of $\tau$ and give a list of lemmas that will be useful to finally prove Theorem \ref{counterexample}.

For simplicity, we fix $T=+\infty$. Let us consider the same $L>0$ and $(k,\beta, \alpha)\in (0,+\infty)^3$ with the condition $\alpha>4 \beta L$. We consider the following $F$: $$F(x)=k+\beta(x-L)^2+\alpha (x-L) \qquad x\in [0,2L],$$
and we continuously extend $F$ to $\mathbb{R}$ by constants; hence $F$ satisfies \eqref{assumptions}. 

We here choose the common reaction time: $\tau=\frac{\pi}{4\alpha}$.

Let us consider $A\in (0,\frac{L}{2})$. We now introduce the initial positions of the vehicles:
\begin{align} \label{initi}
 x_{i}^0(t)=i L+(-1)^i\frac{A}{2} \sin (2\alpha t) \qquad (i,t)\in \mathbb{Z}\times [-\tau,0].
\end{align}
\begin{remark} \label{ini}
We have $x_{i+2}^0-x_{i}^0=2L $ but $x_{i+1}^0-x_{i}^0 \neq L$. We recover the initial data of Example \ref{unstable} for $A=0$. This means that in this case, the vehicles alternatively oscillate around the previous stationnary positions at initial times. Those oscillations will be essential in the construction of the counter-example: they will last for all time and will raise the time velocity of the vehicles that will become striclty superior to the value of the velocity function of the corresponding macroscopic space gradient. 
\end{remark}
The key relationship between $\tau$ and $\alpha$ will allow the periodic oscillations to remain for all times as expressed in the following lemma: 
\begin{lemma}\label{osci}
Under the previous initial conditions, we have:
$$X_{i+1}(t)-X_i(t)=L+A (-1)^{i+1}  \sin (2\alpha t) \qquad t\geq 0.$$
\end{lemma}
\begin{remark}
We have $X_{i+2}-X_i=2L$ for all time. 
\end{remark}
\begin{proof}
The proof is based on the incremental construction of the solutions and is thus an induction proof. We show the result for each $[n\tau,(n+1)\tau], n\in\mathbb{N}$.  We only perform the first step as the next ones are identical. 

Let us define: $d_i:=X_{i+1}-X_i$. 

For $t\in[0,\tau]$, we have:
\begin{align*}
d_i'(t)=F(X_{i+2}(t-\tau)-X_{i+1}(t-\tau))-F(d_i(t-\tau)).
\end{align*}
As $t-\tau\in[-\tau,0]$: 
\begin{align*}
d_i'(t)=F(x_{i+2}^0(t-\tau)-x_{i+1}^0(t-\tau))-F(d_i(t-\tau)).
\end{align*}
Thanks to Remark \ref{ini}, we get:
\begin{align*}
d_i'(t)=F(2L-d_i(t-\tau))-F(d_i(t-\tau)).
\end{align*}
With the expression of $F$, this equivalently gives for $\bar{d_i}:=d_i-L$:
\begin{align*}
\bar{d_i}'(t)=-2\alpha \bar{d_i}(t-\tau).
\end{align*}
Thanks to \eqref{initi}, we have:
\begin{align*}
\bar{d_i}'(t)=-2\alpha A (-1)^{i+1}  \sin (2\alpha t-2\alpha\tau).
\end{align*} 
We remind that we have chosen $\tau=\frac{\pi}{4\alpha}$ to get: 
\begin{align*}
\bar{d_i}'(t)=2\alpha A (-1)^{i+1} \cos (2\alpha t).
\end{align*}
We then integrate:
\begin{align*}
\bar{d_i}(t)=A (-1)^{i+1}  \sin (2\alpha t).
\end{align*}
This gives the result for $n=0$. 
\end{proof}
For $\varepsilon>0$, we now consider $u^\varepsilon$ associated to this initial data and we recall the relationship: 
\begin{align}\label{relationmicro}
X_i(t)=\frac{u^\varepsilon(i\varepsilon,t\varepsilon)}{\varepsilon} \qquad(i,t)\in\mathbb{Z}\times [-\tau,+\infty).
\end{align}
Thanks to the previous remark, we have: 
\begin{align*}
\frac{u^\varepsilon((i+2)\varepsilon,t\varepsilon)-u^\varepsilon(i\varepsilon,t\varepsilon)}{2\varepsilon}=L \qquad (i,t)\in\mathbb{Z}\times [-\tau,+\infty),
\end{align*}
or equivalently: 
\begin{align*}
\frac{u^\varepsilon((i+2)\varepsilon,s)-u^\varepsilon(i\varepsilon,s)}{2\varepsilon}=L \qquad (i,s)\in\mathbb{Z}\times [-\varepsilon\tau,+\infty).
\end{align*}
For $s=0$, this gives for the initial condition: 
\begin{align} \label{gradini}
\frac{u_0((i+2)\varepsilon,0)-u_0(i\varepsilon,0)}{2\varepsilon}=L \qquad (i,\varepsilon)\in\mathbb{Z}\times (0,+\infty).
\end{align}
Moreover, thanks to \eqref{initi}, we have:
$$x_0^0(0)=0.$$
Thus, we also have:
\begin{align}\label{zeroini}
u_0(0,0)=0.
\end{align} 
This leads to the following lemma:
\begin{lemma}
The initial condition is given on the real line by: $u_0(x,0)=L x$.
\end {lemma}
\begin{proof}
Let us consider $\varepsilon>0$, $h>0$ and $x\in\mathbb{R}$. We define:
\begin{align}
\left\{
\begin{array}{l}
  n:=\lfloor \frac{h}{2\varepsilon} \rfloor, \\
  m:=\lfloor \frac{x}{2\varepsilon} \rfloor.
\end{array}
\right.
\end{align}
Let us now compute the rate of change:
\begin{multline}\label{rateini}
\frac{u_0(x+h,0)-u_0(x,0)}{h}=\\
\frac{u_0(x+h,0)-u_0(2(m+n)\varepsilon,0)}{h}+\frac{u_0(2(m+n)\varepsilon,0)-u_0(2m\varepsilon,0)}{h}+\frac{u_0(2m\varepsilon,0)-u_0(x,0)}{h}.
\end{multline}
We introduce a telescopic sum for the second term:
\begin{align*}
\frac{u_0(2(m+n)\varepsilon,0)-u_0(2m\varepsilon,0)}{h}=\sum_{i=0}^{n-1} \frac{u_0(2m\varepsilon+2(i+1)\varepsilon,0)-u_0(2m\varepsilon+2i\varepsilon,0)}{h},
\end{align*}
or equivalently: 
\begin{align*}
\frac{u_0(2(m+n)\varepsilon,0)-u_0(2m\varepsilon,0)}{h}=\frac{2\varepsilon}{h}\sum_{i=0}^{n-1} \frac{u_0(2m\varepsilon+2i\varepsilon+2\varepsilon,0)-u_0(2m\varepsilon+2i\varepsilon,0)}{2\varepsilon}.
\end{align*}
Thanks to \eqref{gradini}, we get: 
\begin{align*}
\frac{u_0(2(m+n)\varepsilon,0)-u_0(2m\varepsilon,0)}{h}=\frac{2n\varepsilon}{h}L.
\end{align*}
Thus, \eqref{rateini} becomes: 
\begin{align}\label{rateini2}
\frac{u_0(x+h,0)-u_0(x,0)}{h}=\frac{u_0(x+h,0)-u_0(2(m+n)\varepsilon,0)}{h}+\frac{2n\varepsilon}{h}L+\frac{u_0(2m\varepsilon,0)-u_0(x,0)}{h}.
\end{align}
Let us remind that: 
\begin{align}\label{estim}
\left\{
\begin{array}{l}
  |h-2n\varepsilon|\leq 2\varepsilon, \\
  |x-2m\varepsilon|\leq 2\varepsilon.
\end{array}
\right.
\end{align}
The left-hand side of \eqref{rateini2} does not depend on $\varepsilon$. By taking the limit when $\varepsilon\rightarrow 0$, using \eqref{estim} and the fact that $u_0$ is a $L-$Lipschitz continuous function, we get:
\begin{align*}
\frac{u_0(x+h,0)-u_0(x,0)}{h}=L.
\end{align*}
Thanks to \eqref{zeroini}, we obtain the desired result on $[0,+\infty)$:
\begin{align*}
u_0(h,0)=Lh \qquad h\geq 0.
\end{align*}
For $h<0$, we define $h':=-h>0$, and we have for $x=h$:
\begin{align*}
\frac{u_0(h+h',0)-u_0(h,0)}{h'}=L,
\end{align*}
and thus:
\begin{align*}
u_0(h,0)=-Lh'=Lh.
\end{align*}
\end{proof}
\begin{corollary} \label{solumacro}
The solution of the macroscopic equation \eqref{exo2} corresponding to this initial data is: $$u^0(x,t)=Lx+F(L)t \qquad (x,t)\in\mathbb{R}\times[0,+\infty).$$
\end{corollary}
\begin{proof}
This is a classical solution and the solution is unique. 
\end{proof}
\begin{remark}
We see that we have: $$\partial_t u^0(0,t)=k=F(L) \qquad t>0.$$

That was also the case for $u^\varepsilon$ in Example \ref{unstable} but it is not true anymore for $A>0$.

Indeed, we have for $s>0$:
\begin{align*}
\partial_t u^\varepsilon (0,s)=F\left(\frac{u^\varepsilon (\varepsilon,s-\varepsilon\tau)-u^\varepsilon (0,s-\varepsilon\tau)}{\varepsilon}\right)
\end{align*}
Thanks to \eqref{relationmicro}, we get: 
\begin{align*}
\partial_t u^\varepsilon (0,s)=F\left(X_1\left(\frac{s}{\varepsilon}-\tau\right)-X_0\left(\frac{s}{\varepsilon}-\tau\right)\right).
\end{align*}
From Lemma \ref{osci}, we get for the velocity function $F$ considered:
\begin{align*}
\partial_t u^\varepsilon (0,s)=F(L)+\beta A^2 \sin^2\left(2\alpha\left(\frac{s}{\varepsilon}-\tau\right)\right)-\alpha A \sin\left(2\alpha\left(\frac{s}{\varepsilon}-\tau\right)\right),
\end{align*} 
or equivalently, for $\tau=\frac{\pi}{4\alpha}$:
\begin{align}\label{derivative}
\partial_t u^\varepsilon (0,s)=F(L)+\beta A^2 \cos^2\left(2\alpha \frac{s}{\varepsilon}\right)+\alpha A \cos\left(2\alpha\frac{s}{\varepsilon}\right).
\end{align} 
\end{remark}
We are now ready to prove Theorem \ref{counterexample}.
\begin{proof} [Proof of Theorem \ref{counterexample}]
By contradiction, we suppose that $u^\varepsilon$ converges locally uniformly towards the solution $u^0$ of \eqref{exo2} whose expression is given in corollary \ref{solumacro}.

Let us consider $s>0$ and $h>0$. We then have:
\begin{align}\label{limitmacro1}
\lim_{\varepsilon\rightarrow 0} \frac{u^\varepsilon(0,s+h)-u^\varepsilon(0,s)}{h}=\frac{u^0(0,s+h)-u^0(0,s)}{h}=F(L).
\end{align}
Otherwise, we have for $\varepsilon>0$:
\begin{align*}
\frac{u^\varepsilon(0,s+h)-u^\varepsilon(0,s)}{h}=\frac{1}{h}\int_s^{s+h} \partial_t u^\varepsilon (0,s') ds'.
\end{align*}
From \eqref{derivative}, we get:
\begin{align*}
\frac{u^\varepsilon(0,s+h)-u^\varepsilon(0,s)}{h}=F(L)+\beta \frac{A^2}{2}+\frac{\beta\varepsilon A^2}{8\alpha h} \left [\sin \left(\frac{4\alpha s'}{\varepsilon}\right)\right]_s^{s+h}+\frac{\varepsilon A}{2h} \left [\sin \left(\frac{2\alpha s'}{\varepsilon}\right)\right]_s^{s+h}
\end{align*}
As $|\sin|\leq 1$, we immediately get: 
\begin{align*}
\lim_{\varepsilon\rightarrow 0} \frac{u^\varepsilon(0,s+h)-u^\varepsilon(0,s)}{h}=F(L)+\beta \frac{A^2}{2},
\end{align*}
which is in contradiction with \eqref{limitmacro1}.
\end{proof}
\textbf{Acknowledgements.}
I would like to thank Cyril Imbert for all his support, help and remarks on this article. I would also like to thank R\'egis Monneau for all his contribution. This work was partly supported by Labex Bézout, (ANR-10-LABX-58-01) and ANR HJnet, (ANR-12-BS01-0008-01).

\bibliographystyle{siam}
\bibliography{biblio}

\end{document}